\documentclass{article}

\usepackage[backend=bibtex, 
sortcites=true,
firstinits=true,
hyperref,
maxbibnames=99,
]{biblatex}

\bibliography{Bibliography}{}

\usepackage{enumerate}

\usepackage[centertags]{amsmath}
\usepackage{amsfonts}
\usepackage{amssymb}
\usepackage{amsthm}
\usepackage{newlfont}
\usepackage{mathtools}
\usepackage[top=1in, bottom=1in, left=1in, right=1in]{geometry}
\usepackage{tikz}

\theoremstyle{plain}
\newtheorem{theorem}{Theorem}[section]

\newtheorem{lemma}[theorem]{Lemma}

\theoremstyle{definition}
\newtheorem{definition}[theorem]{Definition}

\newcommand{\R}{\mathbb{R}}
\newcommand{\N}{\mathbb{N}}

\newcommand{\E}{\mathbb{E}}

\newcommand{\Deg}{\operatorname{Deg}}
\newcommand{\Ric}{\operatorname{Ric}}

\newcommand{\Hidden}[1]{}

\begin{document}

\title{Ollivier curvature, betweenness centrality and average distance}
\author{

Florentin M\"unch\footnote{Max Planck Institute for Mathematics in the Sciences Leipzig, muench@mis.mpg.de}
}
\date{\today}
\maketitle

\begin{abstract}  
We give a new upper bound for the average graph distance in terms of the average Ollivier curvature.
 Here, the average Ollivier curvature is weighted with the edge betweenness centrality.
Moreover, we prove that equality is attained precisely for the reflective graphs which have been classified as Cartesian products of cocktail party graphs, Johnson graphs, halved cubes, Schläfli graphs, and Gosset graphs.
\end{abstract}


\section{Introduction}

A major question in network analysis is how to evaluate the importance of a given edge or vertex. Global measures of importance are called centrality, and local measures of importance can be described via several discrete curvature notions. For example, the clustering coefficient can be related to Ollivier curvature \cite{jost2014ollivier}, and local connectivity is related to Bakry Emery curvature \cite[Section~1.3]{cushing2016bakry}. Moreover, negative curvature indicates a local bottleneck as discussed in \cite{topping2021understanding}.

In contrast to curvature, most centrality notions of an edge or vertex take the whole graph into account. Particularly, the betweenness centrality counts the geodesics passing through a given edge or vertex.

For an overview about different centrality notions and applications, see \cite{borgatti2005centrality,koschutzki2005centrality,
freeman1978centrality}, and for an overview about different curvature notions, see
\cite{ollivier2009ricci,lin2010ricci,forman2003bochner,
jost2021characterizations}.
For a comparison between centrality and curvature in network analysis, see \cite{sreejith2016forman,ni2015ricci,gao2019measuring}.
Recently, there have also been introduced several non-local curvature notions which might be interpreted as some centrality index \cite{steinerberger2022curvature,devriendt2022discrete}.

While curvature and centrality both measure the importance of a vertex or edge, just on a different spatial scale, there seems to be no mathematical connection in the literature so far.

In this paper, we give a connection between edge betweenness centrality and Ollivier Ricci curvature. 
We are particularly interested in the weighted average of the Ollivier Ricci curvature, where the weight is the edge betweenness centrality. Intuitively, this means that important edges contribute more to the average than less important edges.

The main goal of the paper is to give an upper bound of the size of the network in terms of the average Ollivier Ricci curvature.
Moreover for aficionados, we characterize the graphs for which the upper bound is attained.

We now present our main theorem.
\begin{theorem}\label{thm:Main}
Let $G=(V,E)$ be a finite connected graph. Then,
\[\E_g \Ric \cdot \E d \leq \E \Deg\]
where $\E d$ is the average distance, $\E \Deg$ is the average vertex degree, and $\E_g \Ric$ the average Ollivier Ricci curvature weighted with the edge betweenness centrality $g$. 

Moreover, t.f.a.e.:
\begin{enumerate}[(i)]
\item $\E_g \Ric \cdot \E d = \E \Deg$,
\item 
The graph $G$ is a Cartesian product of graphs from the following list:
\begin{itemize}
\item Cocktail party graphs,
\item Johnson graphs,
\item Halved cubes,
\item Schläfli graph,
\item Gosset graph.
\end{itemize}
\end{enumerate}
\end{theorem}

The definitions are given in Section~\ref{sec:Setup}.
The average distance estimate is proven in Section~\ref{sec:DistanceBound}. 
The graphs in $(ii)$ have been characterized as reflective graphs, see Theorem~\ref{thm:Reflective}, and Section~\ref{sec:Reflective} for the definition of reflective graphs.
The implication $(i) \Rightarrow (ii)$ is proven in Section~\ref{sec:Rigidity}. Finally, the implication $(ii) \Rightarrow (i)$ is proven in Section~\ref{sec:Tensorization}.

\subsection{Related work}
Our estimate is a major improvement of the celebrated discrete Bonnet Myers diameter bound which gives an upper bound to the network size in terms of a lower Ricci curvature bound \cite[Proposition~23]{ollivier2009ricci}.
The key difference is that our result only requires a positive average Ricci curvature while the classical diameter bound needs a positive minimum curvature, and in practice, the average curvature of a real network is much larger than the minimum curvature.

The average curvature has also been used by Paeng in \cite{paeng2012volume} to give a diameter bound. However, Paeng only gets meaningful estimates if the curvature is at least half the maximum possible curvature.
Non-constant curvature bounds in the spectral sense and Kato bounds have been employed to estimate the spectral gap and to show finiteness of the the fundamental group \cite{munch2020spectrally}.
Distance bounds allowing some negative curvature have been shown in \cite{liu2017distance} and improved in \cite{munch2018perpetual}.

We now discuss related rigidity results.
Investigating equality in diameter bounds is a fruitful subject and has revealed connections between curvature and interesting classes of graphs \cite{liu2017rigidity,cushing2018rigidity,munch2022reflective,
kamtue2020bonnet}.
Our rigidity result is an improvement of the main result in \cite{munch2022reflective}.
\begin{theorem}[{\cite[Theorem~1.1, Theorem~4.2, and Theorem~4.3]{munch2022reflective}}]\label{thm:Reflective}
Let $G=(V,E)$ be a connected graph with $\Ric >0$. Then,
\begin{enumerate}[(a)]
\item
$
\E d  \leq \frac{\max \Deg}{ \min \Ric},
$
\item Equality holds if and only if $G$ is reflective and has constant curvature,
\item $G$ is reflective if and only if it is a Cartesian product of 
\begin{itemize}
\item Cocktail party graphs,
\item Johnson graphs,
\item Halved cubes,
\item Schläfli graph,
\item Gosset graph.
\end{itemize}
\end{enumerate} 
\end{theorem}
For the definition of reflective graphs, we refer to Section~\ref{sec:Reflective}.
In comparison, our main theorem vastly strengthens the distance bound, and still fully characterizes equality. Surprisingly, we do not get much more graphs with equality. Indeed, the only difference is that in our estimate, equality is also attained if the factors have different curvature.


\section{Setup and notation}\label{sec:Setup}
A finite graph $G=(V,E)$ consists of a finite vertex set $V$ and a symmetric edge relation $E \subset V^2$ with empty diagonal. We write $x\sim y$ if $(x,y) \in V$. The graph distance $d:V^2 \to \N_0$ is given by
\[
d(x,y) := \inf\{n: x=x_0 \sim \ldots \sim x_n=y\}.
\]
We say $G$ is connected if $d < \infty$. The Laplacian $\Delta: \R^V \to \R^V$ is defined as
\[
\Delta f(x) := \sum_{y\sim x}(f(y)-f(x)).
\]
and the vertex degree is defined as 
\[
\Deg(x) := \Delta d(x,\cdot)(x).
\]
We now define betweenness centrality, average vertex degree, average distance, and average Ollivier Ricci curvature.
\begin{definition}[Edge betweenness centrality]
The edge betweenness centrality $g$ of an edge $e$ is defined as
\[
g(e) := \frac 1 {|V|^2}\sum_{x,y} \frac {1}{|P_{xy}|} \sum_{p \in P_{xy}} 1_{e\in p}.
\]
where $P_{xy}$ is the set of shortest paths between $x$ and $y$.
In other words, $g(e)$ is the probability that $e$ is on a uniformly random geodesic between two uniformly random vertices.
\end{definition}

\begin{definition}[Average degree and distance]
The average vertex degree $\Deg_{av}$ is given by
\[
\E\Deg := \frac 1 {|V|}\sum_{x \in V} \Deg(x)
\]
and the average distance by
\[
\E d := \frac 1 {|V|^2} \sum_{x,y \in V} d(x,y).
\]
\end{definition}

\begin{definition}[Average curvature]
The average curvature $\E_g \Ric$ with respect to the betweenness centrality $g$ is given by
\[
\E_g \Ric := \frac{\sum_{e\in E} g(e) \Ric(e)}{\sum_{e\in E} g(e)}
\]
where $\Ric$ is the Ollivier curvature given by
\[
\Ric(x,y) := \inf_{\substack{\|\nabla f\|_\infty = 1\\f(y)-f(x)=d(x,y)}} \frac{\Delta f(x) - \Delta f(y)}{d(x,y)}
\]
with $\|\nabla f\|_\infty := \sup_{x \sim y} |f(y)-f(x)|$, see \cite{munch2017ollivier}.
\end{definition}
We remark that Ollivier and Forman curvature \cite{forman2003bochner} coincide on every edge when maximizing the latter over the choice of 2-cells with given 1-skeleton \cite[Theorem~1.2]{jost2021characterizations}.

\subsection{Reflective graphs}\label{sec:Reflective}
We recall the notion of reflective graphs introduced in \cite{munch2022reflective}.
Let $G=(V,E)$ be a finite graph. Let $x\sim y$. 
We write
\[
V_x^y := \{z \in V: d(x,z) < d(y,z)\}
\]
and
\[
V^{xy} :=  \{z \in V: d(x,z) = d(y,z)\}.
\]
A map $\phi$ is called reflection from $x$ to $y$ if
\begin{enumerate}[(a)]
\item $\phi(x)=y$,
\item $\phi(z)=z$ for all $z \in V^{xy}$,
\item $\phi^2 = id$,
\item $E(V_x^y,v_y^x) = \{(z,\phi(z)):z \in V_x^y\}$,
\item $\phi$ is a graph automorphism.
\end{enumerate}
A graph $G$ is called reflective if for all $x\sim y$, there exists a reflection from $x$ to $y$.

\section{Distance bound}\label{sec:DistanceBound}
We now prove the average distance bound.
\begin{theorem}\label{thm:DistBound}
 Let $G=(V,E)$ be a finite connected graph. Then,
\[
\E_g \Ric \cdot \E d \leq \E \Deg.
\]
\end{theorem}
\begin{proof}
For all $x \neq y$, and all shortest paths $p \in P_{xy}$, we have with $f:=d(x,\cdot)$,
\begin{align}\label{eq:proofSharp}
\Delta f(x) - \Delta f(y) \geq \Ric(x,y)d(x,y) \geq \sum_{e\in p} \Ric(e).
\end{align}
Averaging over all $p \in P_{xy}$, and all $x,y \in V$ gives
\[
\E \Deg =\frac{1}{|V|^2}\sum_{x,y} \Delta d(x,\cdot)(x) - \Delta d(x,\cdot)(y) \geq \frac 1 {|V|^2}\sum_{x,y} \frac 1 {|P_{xy}|} \sum_{e \in P_{xy}} \kappa(e) = \sum_{e \in E} g(e)\Ric(e) = \E_g \Ric \cdot \sum_{e \in E} g(e).
\]
On the other hand,
\[
\sum_e g(e) =  \frac 1 {|V|^2}\sum_{x,y} \frac {1}{|P_{xy}|} \sum_{p \in P_{xy}} \sum_{e \in E}1_{e\in p} = \frac 1 {|V|^2}\sum_{x,y} \frac {1}{|P_{xy}|} \sum_{p \in P_{xy}} d(x,y) = \E d.
\]
Combining implies the claim of the theorem immediately.  
\end{proof}

We now give a first characterization of equality in the distance bound, exploiting that the only inequality in the proof above is \eqref{eq:proofSharp}.

\begin{lemma}\label{lem:AnalyticSharpnessChar}
Let $G=(V,E)$ be a finite connected graph. T.f.a.e.:
\begin{enumerate}[(i)]
\item $\E_g \Ric \cdot \E d = \E \Deg$,
\item $\Delta d(z,\cdot)(x) - \Delta d(z,\cdot)(y) = \Ric(x,y) $ whenever $d(z,x)<d(z,y)$.
\end{enumerate}
\end{lemma}

\begin{proof}
Assertion $(i)$ is clearly equivalent to equality in \eqref{eq:proofSharp} for all $x,y$ and all shortest paths $p$ from $x$ to $y$.
With $p=(x=x_0\sim \ldots \sim x_n=y)$, we have
\[
\Delta f(x) - \Delta f(y) = \sum_{k=1}^n \Delta f(x_k) - \Delta f(x_{k-1}) = \sum_{k=1}^n \Ric(x_k,x_{k-1}).
\]
Particularly, $(i)$ is equivalent to 
\[
\Delta f(x_k) - \Delta f(x_{k-1}) = \Ric(x_k,x_{k-1})
\]
whenever $d(x,x_k)>d(x,x_{k-1})$.
Renaming $x$ to $z$, and $x_k$ to $y$, and $x_{k-1}$ to $x$ proves equivalence of $(i)$ and $(ii)$ finishing the proof.
\end{proof}

\section{Rigidity}\label{sec:Rigidity}

This section is devoted to prove that equality in the main estimate implies that the graph is reflective (for the definition, see Section~\ref{sec:Reflective}). We thus always assume in this section that
\begin{align}\label{eq:MainEstimateSharp}
\E_g \Ric \cdot \E d = \E \Deg.
\end{align}

\begin{lemma}[Main lemma]\label{lem:EqualityImpliesReflective}
Let $G=(V,E)$ be a finite connected graph. Assume
\[\E_g \Ric \cdot \E d = \E \Deg.\]
Then, $G$ is reflective.
\end{lemma}

The proof of the reflectiveness is structured as follows
\begin{itemize}
\item Determine local structure,
\item Define parallel edges,
\item Prove existence and uniqueness of a parallel edge starting at given vertex,
\item Define $\phi$ via parallel edges,
\item Prove that $\phi$ is a reflection.
\end{itemize}

We now determine the local structure around an edge $x\sim y$.
\begin{lemma}
Assume $\E_g \Ric \cdot \E d = \E \Deg.$ Let $x\sim y$. Then, there is a perfect matching $\psi_{xy}$ from $B_1(x)\setminus B_1(y)$ and $B_1(y) \setminus B_1(x)$, i.e. $\psi$ is bijective and $z \sim \psi_{xy}(z)$ for all $z$ in the domain.
\end{lemma}

\begin{proof}
By Lemma~\ref{lem:AnalyticSharpnessChar}, we have
\begin{align}
\Delta d(z,\cdot)(x) -  \Delta d(z,\cdot)(y) = \Ric(x,y) \label{eq:Deltadz}
\end{align}
for all $z$ with $d(y,z)> d(x,z)$. 
Plugging in $z=x$, we get
\[
\Ric(x,y) = \Deg(x) - \Deg(y) + |B_1(x) \cap B_1(y)|.
\]
Combining with the case $z=y$, we get $\Deg(x)=\Deg(y)$ and

\[
\Ric(x,y) = |B_1(x) \cap B_1(y)|.
\]
By \cite[Lemma~3.3]{munch2022reflective}, this implies that there is a perfect matching between $B_1(x)\setminus B_1(y)$ and $B_1(y) \setminus B_1(x)$, denoted by $\psi_{xy}$. This finishes the proof.
\end{proof}

We now define parallel edges. Recall that $V_x^y=\{z: d(z,x)<d(z,y)\}$ and $V^{xy} = \{z: d(x,z)=d(y,z)\}$.

\begin{definition}[Parallel edges]
Let $x\sim y$ and $x'\sim y'$. We say $(x,y) \parallel (x',y')$ if $V_x^y = V_{x'}^{y'}$ and $V_y^x = V_{y'}^{x'}$.
\end{definition}

We now show that the perfect matching $\psi_{xy}$ provides parallel edges.

\begin{lemma}\label{lem:psiParallel}
Assume $\E_g \Ric \cdot \E d = \E \Deg$.
Let $x\sim y$ and $x' \in B_1(x) \setminus B_1(y)$ and let $y' = \psi_{xy}(x')$.
Then, $V_x^y = V_{x'}^{y'}$ and $V_y^x = V_{y'}^{x'}$.
In other words, $(x,y) \parallel (x',y')$.
\end{lemma}

\begin{proof}
We first prove $"\subseteq"$.
Let $z \in V_x^y$. Then, $d(y,z)>d(x,z)$, and thus with $f=d(z,\cdot)$,
\[
\Delta f(x) - \Delta f(y) = \Ric(x,y).
\] 
We have
\[
\Delta f(x) - \Delta f(y) = |B_1(x) \cap B_1(y)| + \sum_{\widetilde x \in B_1(x)\setminus B_1(y)} f(\widetilde x) +1 - f(\psi(\widetilde x)) 
\]
As $\|\nabla f\|_\infty = 1$, the sum is non-negative, and by the equation above, it must be zero. Hence, every part of the sum is zero implying $f(\psi(x')) = f(x') + 1$, meaning $z \in V_{x'}^{y'}$.
This shows $V_x^y \subseteq V_{x'}^{y'}$. Similarly, $V_y^x \subseteq V_{y'}^{x'}$.

We now prove equality. Let $y'' := \psi_{x'y'}(x)$. Then,
$V_y^x \subseteq V_{y'}^{x'} \subseteq V_{y''}^{x}$.
Particularly, $y \in V_{y''}^{x}$ implying $d(y,y'') <d(y,x)=1$ and hence $y=y''$. By $V_y^x \subseteq V_{y'}^{x'} \subseteq V_{y''}^{x}$, this proves equality and in particular, $V_y^x = V_{y'}^{x'}$. 
Similarly, $V_x^y = V_{x'}^{y'}$. This finishes the proof.
\end{proof}

We now prove existence and uniqueness of a parallel edge starting at a given vertex. By this, we define $\phi_{xy}$ on $V_x^y$.

\begin{lemma}\label{lem:parallelBijection}
Assume $\E_g \Ric \cdot \E d = \E \Deg$.
Let $x\sim y$ and $x' \in V_x^y$.
Then, there is a unique $y' \sim x'$ with $(x,y) \parallel (x',y')$.
We denote $y'$ by $\phi_{xy}(x')$.
\end{lemma}

\begin{proof}
We will use that "$\parallel$" is an equivalence relation.
We show existence first. Let $(x=x_0\sim \ldots,\sim x_n=x')$ be a path. we inductively define $y_0:=y$ and $y_{k+1} := \psi_{x_ky_k}(x_{k+1})$. By Lemma~\ref{lem:psiParallel}, we have $(x_k,y_k) \parallel (x_{k+1},y_{k+1})$, and thus, $(x,y) \parallel (x',y_{n})$ showing existence.
We now show uniqueness. Assume $(x,y) \parallel (x',y')$ and $(x,y) \parallel (x',y'')$. Then, $(x',y')\parallel (x',y'')$.
Particularly, $y'' \in V_{y'}^{x'}$ implying $d(y'',y') < d(y'',x') = 1$ and thus, $y''=y'$. This shows uniqueness and finishes the proof.
\end{proof}

The map $\phi_{xy}$ from the above lemma is defined on $V_x^y$.
We now extend $\phi_{xy}$ to a map on $V$.
\begin{definition}
Let $x\sim y$. We define $\phi_{xy}:V\to V$ via
\[
\phi_{xy}(z) := \begin{cases}
z&: z \in V^{xy}, \\
\phi_{xy}(z) &: z \in V_x^y,\\
\phi_{yx}(z) &: z \in V_y^x. 
\end{cases}
\]
\end{definition}

We finally prove reflectiveness.

\begin{proof}[Proof of the main lemma]
Let $x\sim y$. We  will show that $\phi_{xy}$ is a reflection.

\begin{enumerate}[(a)]
\item The property $\phi_{xy}(x)=y$ is clear as $(x,y) \parallel (x,y)$.
\item The property $\phi_{xy}(z) = z$ for all $z \in V^{xy}$ is clear by definition.

\item  
We show that $\phi_{xy}^2 = id$. We aim to show $\phi_{xy}^2(z)=z$ for all $z \in V$.
First let $z \in V_x^y$. We notice $(\phi_{xy}(z),z) \parallel (y,x)$ and $(\phi_{xy}(z),\phi_{yx}(\phi_{xy})(z)) \parallel (y,x)$.
By uniqueness of parallel edges from Lemma~\ref{lem:parallelBijection}, this implies $z= \phi^2_{xy}(z)$. The case $z \in V_y^x$ works similarly, and the case $z \in V^{xy}$ is trivial. This shows $\phi_{xy}^2 = id$.

\item 
We now show $E(V_x^y,V_y^x) = \{(z,\phi_{xy}(z)):z\in V_x^y\}$.
We notice that "$\supseteq$" is clear as $z \sim \phi_{xy}(z)$ for all $z \in V_x^y$.
Now suppose there is $(x', y'') \in E(V_x^y,V_y^x)$ with $y'' \neq \phi_{xy}(x')=:y'$, for which we will find a contradiction.
Then, $y'' \in V_y^x=V_{y'}^{x'}$. Hence, $d(y'',y') < d(y'',x')=1$ implying $y''=y'$. This shows $E(V_x^y,V_y^x) = \{(z,\phi_{xy}(z)):z\in V_x^y\}$.
\item  
We finally show that $\phi:=\phi_{xy}$ is a graph automorphism. By $\phi^2= id$, it suffices to prove $\phi(v)\sim \phi(w)$ whenever $v \sim w$.
As $\phi$ is bijective, it suffiecient to show $d(\phi(v),\phi(w))<2$
Let $v\sim w$. We proceed by case distinction. 
\begin{description}
\item[Case 1:] $v \in V_x^y$ and $w \in V_y^x$. Then, 
$w= \phi(v)$ implying $\phi(v)\sim \phi(w)$.

\item[Case 2:] $v,w \in V_x^y$. 
Then, $\phi(w) \in V_y^x=V_{\phi(v)}^v$ giving $d(\phi(v),\phi(w))< d(v,\phi(w)) \leq 2$.

\item[Case 3:] $v,w \in V^{xy}$. Then clearly, $\phi(v)=v \sim w = \phi(w)$.

\item[Case 4:] $v\in V_x^y$ and $w \in V^{xy}$. Then, $w \in V^{v\phi(v)}$ implying $d(\phi(w),\phi(v)) = d(w,\phi(v)) = d(w,v)=1$.
\end{description}
The remaining cases are analogous, and thus we have shown that $\phi_{xy}$ is a graph automorphism.
\end{enumerate}

In summary, we have proven that $\phi$ is a reflection, finishing the proof.
\end{proof}

\section{Tensorization}\label{sec:Tensorization}

In this section, we aim to show the implication $(ii) \Rightarrow (i)$ in Theorem~\ref{thm:Main}.
We first show that equality in our distance estimate is compatible with Cartesian products.

\begin{lemma}\label{lem:Cartesian}
Let $G_1,G_2$ be finite connected graphs satisfying $\E_g \Ric \cdot \E d = \E \Deg$. Then, $G_1 \times G_2$ satisfies $\E_g \Ric \cdot \E d = \E \Deg$
\end{lemma}

\begin{proof}
We use the characterization from Lemma~\ref{lem:AnalyticSharpnessChar}.
Let $x,y,z \in G_1 \times G_2$ with $d(z,x)<d(z,y)$. Without obstruction, we assume $x=(x_1,x_2)$ and $y=(x_1,y_2)$, and $z=(z_1,z_2)$. 
Then, $d(z,x)<d(z,y)$ implies
$d_2(z_2,x_2)<d_2(z_2,y_2)$.
We observe that
\[
\Delta d(z,\cdot)(x) = \Delta_1 d_1(z_1,\cdot)(x_1) + \Delta_2 d_2(z_2,\cdot)(x_2),
\]
and similarly for $y$. Thus,
\begin{align*}
\Delta d(z,\cdot)(x) - \Delta d(z,\cdot)(y) = \Delta_2 d_2(z_2,\cdot)(x_2) - \Delta_2 d_2(z_2,\cdot)(y_2) = \Ric_2(x_2,y_2)
\end{align*}
as $x$ and $y$ share the first coordinate, and by Lemma~\ref{lem:AnalyticSharpnessChar} for $G_2$.
By \cite[Theorem~3.1]{lin2011ricci}, we have
\[
\Ric(x,y) = \Ric_2(x_2,y_2)
\]
showing that $G_1 \times G_2$ satisfies $(ii)$ in Lemma~\ref{lem:AnalyticSharpnessChar}.
Hence, the proof is finished.
\end{proof}

Having compatibility with Cartesian products, we can now use \cite{munch2022reflective} to prove the implication $(ii) \Rightarrow (i)$ in Theorem~\ref{thm:Main}.

\begin{lemma}\label{lem:ReflectiveImpliesEquality}
Every reflective graph $G$ satisfies 
$\E_g \Ric \cdot \E d = \E \Deg$.
\end{lemma}

\begin{proof}
By Theorem~\ref{thm:Reflective}$(c)$, and as all graphs in the list of the theorem are edge transitive, we know that every reflective graph is a cartesian product of reflective graphs with constant curvature.
By Theorem~\ref{thm:Reflective}$(b)$, every factor of $G$ satisfies 
$\min \Ric \cdot \E d = \max \Deg$, and by Theorem~\ref{thm:DistBound}, this implies 
$\E_g \Ric \cdot \E d = \E \Deg$ for every factor. By Lemma~\ref{lem:Cartesian}, this implies that $G$ satisfies $\E_g \Ric \cdot \E d = \E \Deg$. This finishes the proof. 
\end{proof}

\section*{Acknowledgments}
The author wants to thank Simon Puchert for pointing out that the estimate in Theorem~\ref{thm:Reflective}$(a)$ also works with the average degree instead of the maximum degree, which gave the inspiration to also replace the minimum curvature by some average curvature. The author also wants to thank David Cushing and Norbert Peyerimhoff for useful discussions about reflective graphs during his stay at Newcastle University in July 2022.

\printbibliography

\end{document}